\documentclass[12pt]{amsart}
\usepackage{amssymb}
\usepackage{amsmath}
\usepackage{amsrefs}
\usepackage{hyperref}
\usepackage{bbm}
\usepackage{xcolor}

\definecolor{refkey}{gray}{.5}   
\definecolor{labelkey}{gray}{.5} 
\usepackage{graphicx}
\usepackage{tikz-cd}
\usepackage{epsfig}
\usepackage{mathtools,amsfonts}
\usepackage{here}

\usepackage{comment}
\theoremstyle{plain}

\newtheorem{theorem}{Theorem}[section]
\newtheorem{lemma}[theorem]{Lemma}

\newtheorem{corollary}{Corollary}
\theoremstyle{definition}

\newtheorem*{question*}{Question}

\numberwithin{equation}{section}

\numberwithin{equation}{section}

\allowdisplaybreaks 

\title[Certain results on uniform circle random covering problems]{Certain results on uniform circle random covering problems}

\begin{document}

\author{Davit Karagulyan}

\date{}

\maketitle
\begin{abstract}
In this note we extend a theorem from \cite{Koivusalo-Liao-Persson}. 
\end{abstract}

Let $L=\{ \ell_{n}\}_{n\in \mathbb{N}}$ be a decreasing sequence of positive numbers with $0<\ell_{n}<1$ and $\left\{\omega_{n}\right\}_{n \in \mathbb{N}}$ a sequence of i.i.d. random variables on the circle $\mathbb{T}=\mathbb{R}/\mathbb{Z}$. 
Consider the random intervals
\[
B(\omega_{n},r_{n}):=\left(\omega_{n}-r_{n}, \omega_{n}+r_{n}\right),
\]
where $r_n=\ell_{n} / 2$. 
If $\left\{\omega_{n}\right\}$ are uniformly distributed on $\mathbb{T}$, the Borel-Cantelli lemma assures that Lebesgue almost every (a.e. for short) point on $\mathbb{T}$ is covered infinitely often with probability one if and only if $\sum_{n=1}^{\infty} \ell_{n}=\infty$. 
Moreover, $\sum_{n=1}^{\infty} \ell_{n}<\infty$ implies that Lebesgue a.e. point in $\mathbb{T}$ is covered finitely often with probability one. 
In \cite{Dvoretzky1956}, Dvoretzky  asked the question whether $\mathbb{T}= \limsup_{n \rightarrow \infty} B(\omega_{n},r_{n})$ almost surely. 
This condition means that every point on $\mathbb{T}$ is covered infinitely often with the intervals $B(\omega_{n},r_{n})$.
Kahane \cite{Kahane1959} proved that $\mathbb{T}=\limsup_{n\to \infty } B(\omega_{n},r_{n})$ a.s. when $\ell_{n}=c/n$ with $c>1$, but this is not the case with $0<c<1$ as was proved by Billard \cite{Billard1965}. 
The case $\ell_{n}=1/n$ was solved by Mandelbrot \cite{Mandelbrot1972}. 
A complete answer was obtained by Shepp \cite{Shepp1972}: $\mathbb{T}=\limsup_{n\to \infty} B(\omega_{n},r_{n})$ a.s. if and only if
\[
\sum_{n=1}^{\infty} \frac{1}{n^{2}} \exp \left(\ell_{1}+\cdots+\ell_{n}\right)=\infty .
\]

In \cite{Koivusalo-Liao-Persson}, the authors introduce a new circle-covering method called uniform random covering. 
One can think of this as a uniform version of the Dvoretzky covering problem. 
Suppose that $\{\omega_n\}$ is an i.i.d. sequence uniformly distributed on $\mathbb{T}$.  
For every $n\in \mathbb{N}$, let $E_{n}=\cup_{k=1}^{n} B(\omega_{k}, r_{n})$. 
Define then
\[
\mathcal{U}(\omega)=\liminf _{n \rightarrow \infty} E_{n}.
\]
Notice that for every $y\in \mathcal{U}(\omega)$, there is $N\in \mathbb{N}$ such that for every $n\geq N$ there exists $k\in \{1,\dots ,n\}$ such that $|y-\omega_{k}|<r_n$. 
Hence, we say there is a uniform random covering of the set $A\subset \mathbb{T}$ if $\Pr (\omega\colon A\subset \mathcal{U}(\omega))=1$. 
In \cite{Koivusalo-Liao-Persson}, the authors investigate the almost sure covering properties for such coverings. 
In particular they investigate the probabilities $\Pr(\mathcal{U}(\omega)=\mathbb{T})$. 
In this paper we will study the probabilities $\Pr(\mathcal{U}(\omega)=A)$ for general sets $A$. 
We will also discuss some stability results for such coverings.

We next prove a theorem that extends \cite{Koivusalo-Liao-Persson}*{Theorem 1} to general sets $A\subset \mathbb{T}$. 
We use techniques from the theory of Dvoretzy covering problem. Before formulating the theorem, we would like to point out a typo in \cite{Koivusalo-Liao-Persson} in order to avoid confusion between notations $r_n$ and $\ell_n$. The notation $r_n$ in \cite{Koivusalo-Liao-Persson}, is defined as the radius of the intervals but sometimes is used as its length. For instance, in \cite{Koivusalo-Liao-Persson}*{Theorem 1}, $r_n$ is the length of the intervals. 
This corresponds to $\ell_n$ in our paper and not $r_n=\ell_n /2$. 
We remark that the proof provided in \cite{Koivusalo-Liao-Persson} relies on formulas that hold only for random intervals of the same length and for the entire circle. 

\begin{theorem}\label{M1}
Let $L=\{\ell_n\}$ be a decreasing sequence of positive numbers with $\ell_n\in (0,1)$, and $\{\omega_n\}$ be an i.i.d. sequence uniformly distributed on $\mathbb{T}$. 
\begin{enumerate}
    \item[1)] Let  
    \[
    \delta=\liminf_{n\rightarrow \infty} \left(\frac{n \ell_n}{\ln n}\right).
    \]
    Assume that $\delta<1$ and let $A \subset \mathbb{T}$ be a set with $\delta<\dim_{\mathrm{H}} A$. 
    Then $\Pr (\omega: A\subset \mathcal{U}(\omega))=0$. 
    Furthermore, the Hausdorff dimension of the random subset of $A$ that is not covered by $\mathcal{U}(\omega)$ has Hausdorff dimension at least $\dim_{\mathrm{H}} A-\delta$ almost surely.
    \item[2)] Let $A\subset \mathbb{T}$ be a set with $\overline{\dim}_{\mathrm{B}} A\leq\beta$ (or $N(r)\leq r^{-\beta}$, where $N(r)$ is the number of balls of size $r$ needed to cover $A$) and assume that for some $d\in (0,1)$ we have that 
    \[
    \sum_{n=1}^\infty\left(\frac{1}{\ell_n}\right)^\beta\exp \left(-n d \ell_n\right)<\infty.
    \]
    Then $\Pr (\omega\colon A\subset \mathcal{U}(\omega))=1$. 
\end{enumerate}
\end{theorem}

Here and below, let $\dim_{\mathrm{H}}$ and $\overline{\dim}_{\mathrm{B}}$ denote the Hausdorff dimension and upper box-counting dimension, respectively. 
We refer the reader to \cites{Falconer1990, Kahane1985}. 

\begin{corollary}\label{cor:dim}
    For the sequence $\ell_n=c\frac{\ln n}{n}$ and the set $A\subset \mathbb{T}$ there is uniform circle covering by $\mathcal{U}(\omega)$ if $c>\overline{\dim}_{\mathrm B} A+1$ and no covering if $c<\dim_{\mathrm H} A$.
\end{corollary}

\begin{corollary}\label{cor:covering_fail}
If $\delta<1$ and $A=\mathbb{T}$, then there is no $f$ such that $\Pr (\omega\colon A\subset \mathcal{U}(\omega,f))=1$: namely there is no $f \in L^1(\mathbb{T})$ such that if we sample the centers with respect to $f$, then we will have a uniform random covering of the circle. 
Hence, in this case, the uniform random covering can not be recovered by a perturbation.
\end{corollary}

\subsection*{Acknowledgement}
I would like to thank Tomas Persson, Michihiro Hirayama for helpful discussions. 

\section{Proof of Theorem \ref{M1} and Corollaries \ref{cor:dim} and \ref{cor:covering_fail}}

\subsection{Proof of Theorem \ref{M1}-1) and Corollary \ref{cor:covering_fail}}

\begin{proof}[Proof of Theorem \ref{M1}-1)] 

Let $\{n_k\}_{k\geq 1}$ be an increasing sequence of integers and write
\[
E_{n_k}(\omega)=\cup_{s=1}^{n_k} B_{s,n_k}= \left(\cup_{s=1}^{n_{k-1}}B_{s,n_k}\right) \cup  \left(\cup_{s=n_{k-1}+1}^{n_{k}}B_{s,n_k}\right)=R_{n_k}\cup C_{n_k},
\]
where $B_{s,n_k}=B(\omega_s,r_{n_k})$ with $r_{n_k}=\ell_{n_k}/2$. 
Note that if $n_k$ growth fast enough, then for arbitrary $\alpha<1$ we will have for $R_{n_k}$  that
\[
\sum_{k=1}^\infty \sum_{s=1}^{n_{k-1}}|B_{s,n_{k}}|^\alpha=\sum_{k=1}^\infty n_{k-1} \ell_{n_{k}}^\alpha<\infty.
\]
Thus, for arbitrary $\omega$ the set of points that is covered infinitely often by the sets $R_{n_k}$, i.e. $\limsup_{k\to \infty} R_{n_k}$ will have zero Hausdorff dimension. 
We now assume that $\{n_k\}$ fulfils the above conditions. 
Next note that the blocks $\{ C_{n_k}\}_{k\in \mathbb{N}}$ are independent, then by the zero-one law we will have that the probabilities
\begin{equation}\label{A2}
    \Pr \left(A\subset\limsup_{k\rightarrow \infty}C_{n_k}\right)=\Pr \left(A\subset\limsup_{k\rightarrow \infty}\bigcup_{s=n_{k-1}+1}^{n_{k}}B_{s,n_k}\right)
\end{equation}
is either zero or one. 
We now consider the sequence of random intervals from above, namely $\{B_{s,n_k}\colon n_{k-1}+1\leq s \leq n_{k}, k\in \mathbb{N}\}$.
Note if there is almost eventual covering then there is Dvoretzky covering for the above sequence of random intervals.
Define a new sequence of lengths as follows 
\[
\ell'_s=\ell_{n_{k+1}} 
\]
for all $n_{k}< s\leq n_{k+1}$. 
These are simply the intervals in each block $B_{s,n_k}$. 
Let $L'=\{\ell'_s\}_{s\in \mathbb{N}}$. 
Recall that we let $\delta=\liminf_{n\to \infty} (n\ell_n)/\ln n$. 
\begin{lemma}\label{lem:D<c} 
One can chose the sequence $\{n_k\}_{k}$ in such a way that 
\[
D(L')\leq \delta.
\]
\end{lemma}
\begin{proof}
Let $x$ be a number so that $n_k< n_k +x \leq n_{k+1}$. 
Consider
\[
\frac{\sum_{s\leq n_{k}}\ell'_s + x \ell'_{n_{k+1}}}{\ln (n_k + x)}.
\]
Set $a=\sum_{s\leq n_{k}}\ell'_s$.
We want to show that for all $x$ we have
\[
\frac{a + x \ell'_{n_{k+1}}}{\ln (n_k + x)}<\delta.
\]
Alternatively $A(x)=a + x \ell'_{n_{k+1}} -\delta \ln (n_k+x)<0$. 
For this we compute
\[
\frac{d}{dx}A(x)=A'(x)=\ell'_{n_{k+1}} -\frac{\delta}{n_k+x}.
\]
We see that $A(x)$ has only one critical point. 
For small value of $x$ $A(x)$ decreases, since $\ell'_{n_{k+1}}$ is small, while for large values it increases. 
Hence, we see that the maximum is achieved at the endpoints $A(1)$ and $A(n_{k+1}-n_{k})$. 
Thus, to estimate the maximum of $\limsup_{N \to \infty} (\sum_{s \leq N}\ell'_s/\ln N)$, we need to estimate it at times $\{n_k\}_{k \geq 1}$. 
Let $q=\limsup_{n \to \infty} n \ell'_{n}$. If $q<\infty$, then we will have that $\sum_{s \leq n_k}\ell'_s \leq kq$. Hence
\[
\frac{\sum_{s\leq n_{k}}\ell'_s}{\ln n_k}\leq \frac{k q}{\ln n_k}\to 0
\]
as $k\to \infty$. 
If $q=\infty$ we can take $n_k$ so large that
\[
\limsup_{k \rightarrow \infty} \frac{\sum_{s\leq n_{k}}\ell'_s}{\ln n_k}{{=}}\limsup_{k \rightarrow \infty} \frac{\sum_{s\leq n_{k-1}}\ell'_s+(n_k-n_{k-1})\ell_{n_k}}{\ln n_k}=\liminf_{k \rightarrow \infty}\frac{n_k \ell_{n_k}}{\ln n_k}=\delta.
\]
\end{proof}

We next state \cite{Kahane1985}*{page 160, Theorem 4} on Dvoretzky covering. 
\begin{theorem}\label{Thm-Kh}
Let $A\subset \mathbb{T}$ be a set so that $\alpha=\dim_{\mathrm{H}} A<1$. 
If $\alpha>D(L)$, then $A \not\subset \limsup B(\omega_n,r_n)$ almost surely. 
Moreover, the random subset of $A$ which is not covered infinitely often has Hausdorff dimension at least $\alpha-D(L)$ almost surely.
\end{theorem}
\begin{proof}
The proof follows from \cite{Kahane1985}*{page 160, Theorem 4-2) and 3)}. 
It is proven in the theorem under the assumption that the set $A$ has a finite box-counting dimension. 
However, the proof of the part for no covering is more general and no assumption on the box-counting dimension is needed. It relies on Proposition 7 of the same book which is of a general nature.
\end{proof}

Thus, due to Lemma \ref{lem:D<c} we have $D(L')\leq \delta<\dim_{\mathrm{H}} A$ and Theorem \ref{Thm-Kh} implies $\Pr \left(A\subset\limsup_{n\rightarrow \infty}B(\omega_n,\ell'_n/2)\right)=0$. 
But this implies that in \eqref{A2} we have $\Pr (A\subset\limsup_{k\rightarrow \infty}C_{n_k})=\Pr (A\subset\limsup_{k\rightarrow \infty}\cup_{s=n_{k-1}+1}^{n_{k}}B_{s,n_k})=0$. 
We also have that the Hausdorff dimension of the random set that is not covered infinitely often is of Hausdorff dimension at least $\dim_{\mathrm{H}} A-\delta>0$. 
Since $\dim_{\mathrm{H}}(\limsup_{k\to \infty} R_{n_k})=0$ for every $\omega$, then the set of points that are not covered by $\mathcal{U}(\omega)$ has Hausdorff dimension at least $\dim_{\mathrm{H}} A-\delta>0$.  
This finishes the proof of Theorem \ref{M1}-1).
\end{proof}

\begin{proof}[Proof of Corollary \ref{cor:covering_fail}]
According to the proof of Theorem \ref{M1}-1), it is necessary that $\Pr (A\subset\limsup_{k\rightarrow \infty}C_{n_k})=1$
or that for the sequence $\{\cup_{s=n_{k-1}+1}^{n_{k}}B_{s,n_k}\}_{k=1}^\infty$ there is Dvoretzky covering. 
For any density $f$, different from the uniform, we have that $1/m_f>1$. 
However, by Lemma \ref{lem:D<c} we have that $D(L)<1<1/m_f$. 
Hence, there can be no Dvoretzky covering for $f$ due to \cite{HK2021}*{Theorem 1.1-a)}.
\end{proof}

\subsection{Proof of Theorem \ref{M1}-2)}

Recall that we let $E_n=\cup_{j=1}^{n} B(\omega_j,r_n)=\cup_{j=1}^{n} (\omega_j-r_n,\omega_j+r_n)$. 
Given $\varepsilon>0$, we consider a minimal system of balls of radius $\varepsilon$ covering $A$, say $B\left(x_{1}, \varepsilon\right), \ldots, B\left(x_{N}, \varepsilon\right)$; then $N=N(\varepsilon, A)$ is the $\varepsilon$-covering number of $A$. 
Since $A \subset \cup_{m=1}^{N} B\left(x_{m}, \varepsilon\right)$, then one has
\[
\Pr \left(A \not \subset E_{n}\right) \leq \sum_{m=1}^{N} \Pr \left(B\left(x_{m}, \varepsilon\right) \not \subset E_{n}\right).
\]
Since the distribution of $B(\omega_j,r_n)$ s is invariant under translation, all probabilities in the second member are equal. 
Moreover, one has 
\[
\Pr \left(B(x_m, \varepsilon) \not \subset E_{n}\right)=\Pr\left(x_m \notin \bigcup_{j=1}^{n} B\left(\omega_j,r_n-\varepsilon\right)\right)
\]
for every $m=1,\dots,N$. 
It follows that 
\[
\Pr \left(B(x_m, \varepsilon) \not \subset E_{n}\right)=\prod_{j=1}^{n} \Pr\left(x_m \notin  B\left(\omega_j,r_n-\varepsilon\right)\right)=\prod_{j=1}^{n}(1-(\ell_n-2\varepsilon))
\]
as $\{\omega_n\}$ is uniformly distributed. 
Hence we have 
\[
\Pr \left(A \not \subset E_{n}\right) \leq \sum_{m=1}^{N} \exp \left\{ -\sum_{j=1}^{n} (\ell_n-2\varepsilon)\right\}=\sum_{m=1}^{N} \exp \{-n(\ell_n -2\varepsilon)\}.
\]

We next choose $x\in (0,1)$ and $\varepsilon_n>0$ in such a way that $x\ell_n=2\varepsilon_n$. 
Since $N(\varepsilon,A)\leq \varepsilon^{-\beta}$, we have
\[
\Pr (A \not\subset E_n)\leq \left(\frac{2}{x\ell_n}\right)^\beta\exp \left\{-n (1-x) \ell_n\right\}.
\]
By assumption, we now have that 
\[
\sum_{n=1}^\infty\left(\frac{2}{x\ell_n}\right)^\beta\exp \left\{-n (1-x) \ell_n\right\} <\infty.
\]
for $1-x\in (0,1)$. 
Then, by the Borel-Cantelli lemma we will have that $\Pr(\omega \colon A\subset \mathcal{U}(\omega))=1$.

\subsection{Proof of Corollary \ref{cor:dim}} 

Let $\ell_n=c\frac{\ln n}{n}$. 
It is immediate from Theorem \ref{M1}-1) that there is no covering if $c<\dim_{\mathrm H} A$ since $\delta =c$.

For $\ell_n$, the series in the theorem is equivalent to
\[
\sum_{n=1}^\infty \frac{n^\beta}{(c\ln n)^\beta }\frac{1}{n^{cd}} =\sum_{n=1}^\infty \frac{1}{n^{cd-\beta }(c\ln n)^\beta }.
\]
Thus, we will have convergence if $cd-\beta>1$, that is, $c>(1+\beta)/d$. We see that one can find $d\in (0,1)$ for which the series above is convergent if and only if $c>1+\beta$.

\medskip

Institute of Mathematics, Armenian National Academy of Sciences Marshal Baghramian ave. 24b, Yerevan, 375019, Armenia

Email: dakaragulyan@gmail.com


\begin{thebibliography}{00}

\bib{Akcoglu-Bellow-Jones-Losert-ReinholdLarsson-Wierdl1996}{article}{
   author={Akcoglu, Mustafa},
   author={Bellow, Alexandra},
   author={Jones, Roger L.},
   author={Losert, Viktor},
   author={Reinhold-Larsson, Karin},
   author={Wierdl, M\'{a}t\'{e}},
   title={The strong sweeping out property for lacunary sequences, Riemann
   sums, convolution powers, and related matters},
   journal={Ergodic Theory Dynam. Systems},
   volume={16},
   date={1996},
   number={2},
   pages={207--253},
   issn={0143-3857},
   review={\MR{1389623}},
   doi={10.1017/S0143385700008798},
}

\bib{Akcoglu-Ha-Jones1997}{article}{
   author={Akcoglu, Mustafa A.},
   author={Ha, Dzung M.},
   author={Jones, Roger L.},
   title={Sweeping out properties of operator sequences},
   journal={Canad. J. Math.},
   volume={49},
   date={1997},
   number={1},
   pages={3--23},
   issn={0008-414X},
   review={\MR{1437198}},
   doi={10.4153/CJM-1997-001-3},
}

\bib{Billard1965}{article}{
   author={Billard, P.},
   title={S\'{e}ries de Fourier al\'{e}atoirement born\'{e}es, continues, uniform\'{e}ment
   convergentes},
   language={French},
   journal={Ann. Sci. \'{E}cole Norm. Sup. (3)},
   volume={82},
   date={1965},
   pages={131--179},
   issn={0012-9593},
   review={\MR{0182832}},
}

\bib{Bourgain1988}{article}{
   author={Bourgain, J.},
   title={Almost sure convergence and bounded entropy},
   journal={Israel J. Math.},
   volume={63},
   date={1988},
   number={1},
   pages={79--97},
   issn={0021-2172},
   review={\MR{959049}},
   doi={10.1007/BF02765022},
}

\bib{Dvoretzky1956}{article}{
   author={Dvoretzky, Aryeh},
   title={On covering a circle by randomly placed arcs},
   journal={Proc. Nat. Acad. Sci. U.S.A.},
   volume={42},
   date={1956},
   pages={199--203},
   issn={0027-8424},
   review={\MR{79365}},
   doi={10.1073/pnas.42.4.199},
}

\bib{Falconer1990}{book}{
   author={Falconer, Kenneth},
   title={Fractal geometry},
   note={Mathematical foundations and applications},
   publisher={John Wiley \& Sons, Ltd., Chichester},
   date={1990},
   pages={xxii+288},
   isbn={0-471-92287-0},
   review={\MR{1102677}},
}

\bib{Fan-Karagulyan2021}{article}{
   author={Fan, Aihua},
   author={Karagulyan, Davit},
   title={On $\mu$-Dvoretzky random covering of the circle},
   journal={Bernoulli},
   volume={27},
   date={2021},
   number={2},
   pages={1270--1290},
   issn={1350-7265},
   review={\MR{4255234}},
   doi={10.3150/20-bej1273},
}

\bib{HK2021}{article}{
   author={M. Hirayama},
   author={D. Karagulyan},
   title={On genericity of non-uniform Dvoretzky coverings of the circle},
   eprint={https://arxiv.org/abs/2110.07350}
}



\bib{KO2005}{article}{
   author={A. Olevskii, S Kostyukovsk},
   title={Compactness of families of convolution operators with respect to convergence almost everywhere.},
   journal={Real Analysis Exchange},
   volume={30(2)},
   date={2005},
   pages={755–766},
}


\bib{Kahane1959}{article}{
   author={Kahane, Jean-Pierre},
   title={Sur le recouvrement d'un cercle par des arcs dispos\'{e}s au hasard},
   language={French},
   journal={C. R. Acad. Sci. Paris},
   volume={248},
   date={1959},
   pages={184--186},
   issn={0001-4036},
   review={\MR{103533}},
}

\bib{Kahane1985}{book}{
   author={Kahane, Jean-Pierre},
   title={Some random series of functions},
   series={Cambridge Studies in Advanced Mathematics},
   volume={5},
   edition={2},
   publisher={Cambridge University Press, Cambridge},
   date={1985},
   pages={xiv+305},
   isbn={0-521-24966-X},
   isbn={0-521-45602-9},
   review={\MR{833073}},
}

\bib{Karagulyan2011}{article}{
   author={Karagulyan, G. A.},
   title={On the sweeping out property for convolution operators of discrete
   measures},
   journal={Proc. Amer. Math. Soc.},
   volume={139},
   date={2011},
   number={7},
   pages={2543--2552},
   issn={0002-9939},
   review={\MR{2784819}},
   doi={10.1090/S0002-9939-2010-10829-8},
}

\bib{Koivusalo-Liao-Persson}{article}{
   author={Koivusalo, Henna},
   author={Liao, Lingmin},
   author={Persson, Tomas},
   title={Uniform random covering problems},
   journal={International Mathematics Research Notices},
   volume={},
   date={},
   number={},
   pages={},
   issn={},
   review={},
   doi={10.1093/imrn/rnab272},
}
 

\bib{Mandelbrot1972}{article}{
   author={Mandelbrot, Beno\^{\i}t B.},
   title={On Dvoretzky coverings for the circle},
   journal={Z. Wahrscheinlichkeitstheorie und Verw. Gebiete},
   volume={22},
   date={1972},
   pages={158--160},
   review={\MR{309163}},
   doi={10.1007/BF00532734},
}

\bib{Nagel-Stein1984}{article}{
   author={Nagel, Alexander},
   author={Stein, Elias M.},
   title={On certain maximal functions and approach regions},
   journal={Adv. in Math.},
   volume={54},
   date={1984},
   number={1},
   pages={83--106},
   issn={0001-8708},
   review={\MR{761764}},
   doi={10.1016/0001-8708(84)90038-0},
}

\bib{Shepp1972}{article}{
   author={Shepp, L. A.},
   title={Covering the circle with random arcs},
   journal={Israel J. Math.},
   volume={11},
   date={1972},
   pages={328--345},
   issn={0021-2172},
   review={\MR{295402}},
   doi={10.1007/BF02789327},
}

\bib{Be}{article}{
   author={A. Bellow},
   title={Two problems},
   journal={Proc. Oberwolfach Conference on Measure Theory, Springer Lecture Notes in Math.},
   volume={945},
   date={1987},
}



\bib{Tang2012}{article}{
   author={Tang, JunMin},
   title={Random coverings of the circle with i.i.d. centers},
   journal={Sci. China Math.},
   volume={55},
   date={2012},
   number={6},
   pages={1257--1268},
   issn={1674-7283},
   review={\MR{2925591}},
   doi={10.1007/s11425-011-4338-y},
}

\bib{Yu2003}{article}{
   author={Yu, Jinghu},
   title={Covering the circle with random open sets},
   journal={Real Anal. Exchange},
   volume={29},
   date={2003/04},
   number={1},
   pages={341--353},
   issn={0147-1937},
   review={\MR{2061316}},
   doi={10.14321/realanalexch.29.1.0341},
}

\end{thebibliography}
\end{document}